\documentclass[twoside,11pt]{article}

\usepackage{amsmath, mathtools, amssymb, amsthm}
\usepackage{natbib}
\usepackage{microtype}
\usepackage{enumerate, multicol}
\usepackage{bbm}
\usepackage[linesnumbered,ruled,vlined,algo2e]{algorithm2e}
\usepackage[font=small,labelfont=bf]{caption}
\usepackage{array} 
\usepackage[table]{xcolor}
\usepackage{tikz}
\usetikzlibrary{positioning, arrows, shapes.geometric, calc}
\usepackage{fullpage}
\usepackage[colorlinks=true, linkcolor=blue, urlcolor=blue, citecolor=blue]{hyperref}
\usepackage{algorithm,algorithmic}
\usepackage{breakcites}


\definecolor{gaussianfill}{RGB}{220,230,250}
\definecolor{kssgfill}{RGB}{220,250,230}
\definecolor{usgfill}{RGB}{245,230,250}

\SetKwInput{KwInput}{Input}
\SetKwInput{KwOutput}{Output}
\SetKwRepeat{KwRepeat}{repeat}{until}

\theoremstyle{plain}
\newtheorem{theorem}{Theorem}[section]
\newtheorem{lemma}[theorem]{Lemma}

\newtheorem{remark}[theorem]{Remark}
\newtheorem{definition}[theorem]{Definition}
\numberwithin{equation}{section}

\newcommand{\EE}{\mathbb{E}}

\newcommand{\PP}{\mathbb{P}}

\newcommand{\RR}{\mathbb{R}}

\newcommand{\diam}{\mathrm{diam}}

\newcommand{\argmin}{\mathop{\mathrm{argmin}}}

\def\T{^\top}

\allowdisplaybreaks

\begin{document}

\begin{center}
{\LARGE Sparse Convexification for High-Dimensional Constrained Regression}

{\large
\begin{center}
Matey Neykov
\end{center}}

{Department of Statistics and Data Science, Northwestern University\\
\texttt{mneykov@northwestern.edu}}
\end{center}

\begin{abstract}
We study high-dimensional linear regression under a general symmetric convex constraint.  Rather than imposing a specific sparsity-inducing penalty, we start from an arbitrary sign-symmetric and permutation-invariant convex body $K\subseteq \mathbb R^p$ and construct the sparse convexification hierarchy
\[
K^{(s)} = \operatorname{conv}\{v\in K:\|v\|_0\le s\}.
\]
We propose a penalized least-squares estimator that searches over this hierarchy and adapts to the best sparse convex approximation of the target.  Under standard sub-Gaussian assumptions on the random design and noise, we prove an oracle inequality showing that the estimator adapts to the best
sparse convex approximation of the target.  For an $s$-sparse target, the
result yields a squared-error rate governed by the noise level $\sigma$, and the Gaussian width of the sparse convexification $K^{(s)}$.  The method applies broadly to symmetric norm balls and can be implemented using oracle access to the Minkowski functional of $K$. As a special case, the framework yields a consistency result for the constrained Lasso.
\end{abstract}

\tableofcontents
\section{Introduction}

High-dimensional linear regression is a central problem in modern statistics. One observes
\begin{align*}
    Y = X\beta+\xi,
\end{align*}
where $X\in\mathbb R^{N\times p}$ is a design matrix (in this paper we assume random design), $\xi\in\mathbb R^N$ is noise, and the ambient dimension $p$ may be much larger than the sample size $N$. In this regime, consistent estimation of an arbitrary vector $\beta\in\mathbb R^p$ is impossible without structural assumptions. The most classical such assumption is sparsity. If $\beta$ has at most $s\ll p$ nonzero coordinates, then the effective dimension of the problem is of order $s\log(ep/s)$, rather than $p$.

This observation underlies a vast literature on sparse high-dimensional regression. The Lasso \citep{tibshirani1996regression}, basis pursuit and related $\ell_1$-methods \citep{chen2001atomic,donoho2006compressed,candes2006near}, the Dantzig selector \citep{candes2007dantzig}, and subsequent oracle inequalities for the Lasso and related procedures \citep{bickel2009simultaneous,van2009conditions,buhlmann2011statistics} showed that sparse vectors can be estimated at rates depending only logarithmically on the ambient dimension. More generally, high-dimensional $M$-estimation with decomposable regularizers has provided a powerful framework for studying sparse vectors, low-rank matrices, group sparsity, and other structured models; see, for example, \cite{negahban2012unified}.

Most of this theory relies on a specific regularizer or on decomposability properties tailored to the structure of interest. In contrast, this paper considers a different setting. We begin with an arbitrary convex constraint $K\subset\mathbb R^p$, assumed only to be symmetric in the sense of being invariant under sign changes and coordinate permutations. Equivalently, $K$ is the unit ball of a norm whose Minkowski functional
\begin{align}\label{gauge}
    \|x\|_K := \inf\{t>0:x\in tK\}
\end{align}
satisfies
\begin{align}\label{symmetric:norm}
    \|(x_1,\ldots,x_p)\|_K
    =
    \|(\varepsilon_1 x_{\pi(1)},\ldots,\varepsilon_p x_{\pi(p)})\|_K
\end{align}
for every choice of signs $\varepsilon_j\in\{-1,1\}$ and every permutation $\pi$ of $\{1,\ldots,p\}$. This class includes the usual $\ell_q$-balls, Lorentz and convex weak-$\ell_q$ balls, SLOPE-type norm balls, and many other permutation-invariant constraints. We assume access to $K$ only through its Minkowski functional. Thus our goal is not to exploit a special formula for a particular norm, but to develop a general high-dimensional estimation method for symmetric norm constraints.

The main idea of the paper is to sparsify a convex constraint by convexification. For $1\le s\le p$, define
\begin{align*}
    K^{(s)}
    :=
    \operatorname{conv}
    \left\{
        v\in K:\ \|v\|_0\le s
    \right\}.
\end{align*}
The convex-hull construction defining $K^{(s)}$ has previously been studied in the optimization literature for sign- and permutation-invariant norm balls; see Kim, Tawarmalani, and Richard (2022) \cite{kim2022convexification}, as well as the earlier $k$-support norm construction of Argyriou, Foygel, and Srebro (2012) \cite{argyriou2012sparse} for the special case $K=B_2$.

The set $K^{(s)}$ is the convex hull of the $s$-sparse points of $K$.  It is convex by definition, but it retains the statistical complexity of sparse vectors. We then estimate $\beta$ by penalized least squares over the hierarchy $\{K^{(s)}\}_{s=1}^p$:
\begin{align*}
    (\widehat s,\widehat\beta)
    \in
    \argmin_{\substack{1\le s\le p\\ \nu\in K^{(s)}}}
    \bigg\{
        \frac1N\|Y-X\nu\|_2^2
        &+
        \lambda
        \sigma \bigg(\frac{w(K^{(s)}) \sqrt{\log(e + s)}}{\sqrt{N}} + \frac{w(K^{(s)}) \log(e + s)}{N}\bigg) \nonumber\\
        &+ \lambda  \frac{w^2(K^{(s)}) \log(e + s)}{N}
    \bigg\},
\end{align*}
where $\sigma$ is the noise level (i.e., the sub-Gaussian parameter of $\xi_i$), and $\lambda>0$ is a sufficiently large (i.e., larger than a sufficiently large absolute constant) tuning parameter, and $w(K^{(s)})$ stands for (a computable upper bound of) the Gaussian width of $K^{(s)}$.

The statistical content of the paper is an oracle inequality for this estimator. Informally, under standard sub-Gaussian assumptions on the design and noise, we prove that
\begin{align*}
        \|\widehat\beta-\beta\|_2^2
    \lesssim
    \min_{s : s \in [p]}
    \bigg\{
        \inf_{\nu\in K^{(s)}}
        \|\nu - \beta\|_2^2
        +
        \lambda
        \sigma \bigg(\frac{w(K^{(s)}) \sqrt{\log(e + s)}}{\sqrt{N}} + \frac{w(K^{(s)}) \log(e + s)}{N}\bigg)\\
        +
        \lambda  \frac{w^2(K^{(s)}) \log(e + s)}{N}
    \bigg\},
\end{align*}
up to constants depending only on the design distribution. Therefore, if $\beta\in K$ is $s$-sparse, then
\begin{align}\label{intro:bound}
    \|\widehat\beta-\beta\|_2^2
    \lesssim
    \lambda
        \sigma \bigg(\frac{w(K^{(s)}) \sqrt{\log(e + s)}}{\sqrt{N}} + \frac{w(K^{(s)}) \log(e + s)}{N}\bigg) + \lambda  \frac{w^2(K^{(s)}) \log(e + s)}{N}.
\end{align}

Let $K\subseteq R B_2^p$ for some known radius $R$. The radius $R$ is a deterministic and known geometric quantity associated with the constraint $K$. As we argue in the appendix $w(K^{(s)}) \lesssim d_s \sqrt{s \log (ep/s)}$ where $d_s = \operatorname{diam}_2(K^{(s)}) \leq \operatorname{diam}_2(K) \leq 2R$. Hence the bound on the estimation error squared is meaningful so long as 
\begin{align*}
    R^2 (\sigma^2 \vee 1) \log(e + s)s\log(ep/s) \ll N,
\end{align*}
provided that $\lambda$ is chosen of constant order. 
This condition is meaningful even when $p\gg N$, provided the target vector is sufficiently sparse and the radius and sub-Gaussian parameter of the noise are not too large.

We emphasize that the purpose of this paper is not to prove minimax optimality over every symmetric norm ball. Nevertheless, we show that if it is known that $\beta\in K^{(s)}$, and $N$ is sufficiently large, then our estimator is minimax optimal up to
logarithmic factors.

A key geometric ingredient is the following support-function identity. If $T_s(x)$ denotes the indices of the $s$ largest coordinates of $x$ in absolute value, then
\begin{align*}
    h_{K^{(s)}}(x)
    =
    h_K(x_{T_s(x)}),
\end{align*}
where $h_{K^{(s)}}$ and $h_K$ are the support functions of $K^{(s)}$ and $K$ respectively. This identity uses only sign-symmetry and permutation invariance. It shows that the polar geometry of $K^{(s)}$ can be accessed through the original norm $K$ after keeping only the largest $s$ coordinates. As a consequence, the sparse convexification $K^{(s)}$ is not merely a statistical device; it also has a computationally useful description in terms of the Minkowski functional of $K$.

The framework also gives a new consequence for the constrained Lasso. When $K=R B_1^p$, the sparse convexification hierarchy collapses, since
\begin{align*}
    K^{(s)} = R B_1^p
    \qquad\text{for every } s\ge 1.
\end{align*}
Thus the estimator coincides with the constrained Lasso. Nevertheless, our analysis yields a consistency result for the entire $\ell_1$-ball. In particular, if $\beta\in R B_1^p$, using the oracle inequality with $s = 1$ (as $K^{(s)} = R B_1^p$), then the constrained Lasso satisfies a bound of the form
\begin{align*}
    \|\widehat\beta-\beta\|_2^2
    \lesssim \sigma R \sqrt{\frac{\log (ep)}{N}} + R^2 \frac{\log (ep)}{N},
\end{align*}
under the same design assumptions. To the best of our knowledge, this particular consistency statement for the constrained Lasso has not been explicitly emphasized in the high-dimensional regression literature. The rate agrees with the behavior predicted by the general convex-constrained least-squares theory of \cite{chatterjee2014new}, in particular with the phenomenon described in his Theorem~2.1\footnote{We do remark however that Theorem~2.1 is only valid in the low-dimensional scenario, and concerns the in-sample prediction loss rather than the estimation error. In the low dimensional scenario with $\lambda_{\min}(X\T X) \geq c_0$ as assumed in Theorem~2.1 of \cite{chatterjee2014new} the two losses are equivalent.}. It also resembles a slow rate bound for the Lasso \cite[see Theorem 7.20][e.g.]{wainwright2019high}, although in contrast to the classical slow rate bound this is a guarantee on the estimation error of the constrained Lasso rather than a bound on the prediction loss of the regularized Lasso. Finally, we mention a Lasso bound from \cite{raskutti2011minimax} which is very relevant; specifically under fixed design satisfying certain conditions Theorem 2 (a) of \cite{raskutti2011minimax} shows that $\|\hat \beta - \beta\|_2^2 \lesssim R\sigma \sqrt{\frac{\log (ep)}{N}}$ with high probability. In contrast we work in the random design setting and our bound is in expectation. Moreover, \cite{raskutti2011minimax} show this is the minimax optimal rate over the $\ell_1$ ball. 

The rest of the paper is organized as follows. Section \ref{defs:notation:sec} introduces several definitions and commonly used notation. Section~\ref{highdim:regression:section} zooms in on the regression model and the sparse convexification hierarchy. It also contains our main result and its proof. Section \ref{minimax:sec} shows the minimax optimality of our algorithm under some assumptions. Section~\ref{discussion:section} concludes with limitations and open problems.

\subsection{Definitions and Notation}\label{defs:notation:sec}

We use the convenient notation $[n] = \{1,\ldots, n\}$. For a vector $x\in \mathbb R^p$ and an index set $T\subseteq [p]$, we write $x_T$ for the restriction of $x$ to the coordinates in $T$. When convenient, we identify this restriction with its natural embedding in $\mathbb R^p$, obtained by setting all coordinates outside $T$ equal to zero. The intended meaning will always be clear from context. We denote with $\|\cdot\|_{\operatorname{op}}$ the operator norm of a matrix. With a slight abuse of notation, we denote with $\|\nu\|_0$ the cardinality of the set of non-zero entries of the vector $\nu$. We will occasionally use $\operatorname{dist}(x, C)$ to denote the smallest Euclidean distance from $x$ to a closed convex set $C$. Let $\mathbb S^{p-1}$ denote the unit (Euclidean) sphere on $\RR^p$, while $B_2^p$ be the unit Euclidean ball. We start with a classical definition of sub-Gaussian variables.

\begin{definition}[Orlicz norm and sub-Gaussian variables] Let $X$ be a real valued random variable. The Orlicz $2$-norm of $X$ is defined as
\begin{align*}\|X\|_{\psi_2}
:=
\inf\left\{
c>0:\ \mathbb{E}\exp\left(\frac{X^2}{c^2}\right)\le e
\right\}.
\end{align*}

If in addition $X$ is mean zero, and $\|X\|_{\psi_2} < \infty$ it follows that
\begin{align*}
    \EE \exp(\delta \cdot X) \leq \exp(\delta^2 c_0^2), \mbox{ for all } \delta \in \RR,
\end{align*}
 for some absolute constant $c_0 \asymp c$ (see \cite{vershynin2010introduction}). In that case we say that $X$ is a sub-Gaussian variable with parameter $c_0$.
\end{definition}

We now define the polar body of a convex body $K$.

\begin{definition}[Polar body]
Let $K\subset \mathbb{R}^p$ be a convex body containing the origin. The
polar body of $K$ is defined by
\begin{align*}
    K^\circ
    :=
    \left\{
        y\in \mathbb{R}^p:
        \langle x,y\rangle \le 1
        \text{ for all } x\in K
    \right\}.
\end{align*}
Equivalently,
\begin{align*}
    K^\circ
    =
    \left\{
        y\in \mathbb{R}^p:
        h_K(y)\le 1
    \right\},
\end{align*}
where
\begin{align*}
    h_K(y):=\sup_{x\in K}\langle x,y\rangle
\end{align*}
is the support function of $K$.
\end{definition}

Next we define the notion of a weak membership oracle of a closed convex set.

\begin{definition}[Weak membership oracle]
Let $C\subseteq\mathbb R^p$ be closed and convex. A weak membership oracle for $C$
is an oracle which, on input $(x,\varepsilon)\in\mathbb R^p\times(0,\infty)$,
returns $\mathrm{YES}$ if $x\in C$, returns $\mathrm{NO}$ if
\begin{align*}
    \operatorname{dist}(x,C)>\varepsilon,
\end{align*}
and may return either answer otherwise. Equivalently, the oracle is allowed
to be ambiguous only for points outside $C$ whose Euclidean distance to $C$
is at most $\varepsilon$.
\end{definition}

Finally we give the definition of dyadic entropy numbers, which is crucial in establishing the minimaxity of our procedure.

\begin{definition}[Dyadic Entropy Numbers of a Convex Body] The $k$-th dyadic entropy number of $K$ is given by
\begin{align*}
e_k(K)
:=
\inf\left\{
\varepsilon>0:
N(\varepsilon, K)\le 2^{k-1}
\right\},
\qquad k\ge 1,
\end{align*}
where \(N(\varepsilon, K)\) denotes the smallest number of Euclidean balls of radius \(\varepsilon\) needed to cover \(K\).
\end{definition}

We proceed with our main section below.

\section{High-Dimensional Regression}\label{highdim:regression:section}

In this section we discuss the genuinely high-dimensional case, where $p$ may be much larger than $N$.  We focus on the regression model
\begin{align*}
    Y=X\beta+\xi,
    \qquad
    \beta\in K\subseteq \RR^p,
\end{align*}
where $K$ is the unit ball of a symmetric norm. Suppose that the entries $\xi_1,\dots,\xi_N$
are i.i.d., mean-zero, sub-Gaussian random variables with sub-Gaussian parameter at most $\sigma$. 

Let $\beta \in K \subseteq \RR^p$ where $K$ has a symmetric Minkowski functional as detailed in \eqref{gauge} and \eqref{symmetric:norm}. We suppose that the covariates (i.e. the rows of the matrix $X$): $X_i$ are independent from the noise $\xi_i$.

Suppose further that the covariates $X_i$ are centered\footnote{If the predictors $X_i$ are not centered one can consider $(Y_{2i} - Y_{2i-1}, X_{2i} - X_{2i-1})_{i \in [\lfloor N/2\rfloor]}$ to center the predictors while leaving $\beta$ unchanged. We note that this operation prevents the model from having an intercept.} sub-Gaussian vectors with a well conditioned covariance matrix $\EE X_i X_i\T = \Sigma$, i.e., $\lambda_- \leq \lambda_{\min}(\Sigma) \leq \lambda_{\max}(\Sigma) \leq \lambda_+$ for some $\lambda_-, \lambda_+ > 0$, and sub-Gaussian parameter of constant order, i.e., $\sup_{v \in \mathbb S^{p-1}} \EE \exp(\lambda v\T X_i) \leq \exp(\lambda^2 \zeta^2/2)$ for $\zeta = O(1)$. 
Let $R$ be a number such that $K \subseteq R B^p_2$. (as discussed in Lemma 2.5 \cite{neykov2026fast} this number can always be computed for a symmetric set $K$; even more, if one applies Lemma \ref{lem:sparse-diameter-fundamental} with $s = p$ one can approximate the best possible $R$ value within a $\sqrt{\log p}$ factor).

For $1\le s\le p$, define
\begin{align*}
    S_s:=\{v\in K:\|v\|_0\le s\},
    \qquad
    K^{(s)}:=\operatorname{conv}(S_s),
\end{align*}
We write
\begin{align*}
    \phi(s):=s\log(ep/s),\qquad 1\le s\le p.
\end{align*}
Throughout this section we use the notation
\begin{align*}
    d_s:=\diam(K^{(s)})=\sup_{u,v\in K^{(s)}}\|u-v\|_2.
\end{align*}
Since $K^{(s)}$ is centrally symmetric, its Euclidean radius satisfies
\begin{align*}
    \operatorname{rad}(K^{(s)}):=\sup_{v\in K^{(s)}}\|v\|_2=d_s/2 \leq R.
\end{align*}
The estimator considered in this section is

\begin{align}\label{hd:estimator}
    (\widehat s,\widehat\beta)
    \in
    \argmin_{\substack{1\le s\le p\\ \nu\in K^{(s)}}}
    \bigg\{
        \frac1N\|Y-X\nu\|_2^2
        &+
        \lambda
        \sigma \bigg(\frac{w(K^{(s)}) \sqrt{\log(e + s)}}{\sqrt{N}} + \frac{w(K^{(s)}) \log(e + s)}{N}\bigg) \nonumber\\
        &+ \lambda  \frac{w^2(K^{(s)}) \log(e + s)}{N}
    \bigg\}.
\end{align}
Equivalently, for each $s$ one computes a constrained least-squares estimator over the convex set $K^{(s)}$ and then selects $s$ by the complexity penalty above. We note in passing, that the above procedure does not need to be implemented for all $s$ --- a dyadic grid on $s$ suffices to achieve the same rate. Obviously this leads to computational gains. However, to keep the exposition simple we do not pursue this refinement here. 

Before we state and prove our main theorem, we argue that one can implement $\widehat \beta$ in polynomial time.

\subsection{On computing our estimator}

We begin by stating that any symmetric convex body $K \in \RR^p$ with oracle access to its Minkowski functional is well-balanced, in the sense that one can compute $r, R > 0$ such that $rB_2 \subseteq K \subseteq RB_2$ and $R/r \leq p$ \cite[see Lemma 2.5][for instance]{neykov2026fast}. This property is very useful for computational considerations as will become apparent soon.

We first record the oracle reduction for $K^{(s)}$.  For a vector $z\in\RR^p$, let $T_s(z)$ denote the set of indices of the $s$ largest coordinates of $z$ in absolute value, with deterministic tie-breaking.

\begin{lemma}[Support function and polar of $K^{(s)}$]\label{lem:Ks-polar}
Let $K\subseteq\RR^p$ be centrally symmetric, convex, sign-invariant, and permutation-invariant.  Then, for every $z\in\RR^p$,
\begin{align*}
    h_{K^{(s)}}(z)
    :=
    \sup_{v\in K^{(s)}}\langle z,v\rangle
    =
    \max_{|T|\le s}\|z_T\|_{K^\circ}
    =
    \|z_{T_s(z)}\|_{K^\circ}.
\end{align*}
Consequently,
\begin{align*}
    (K^{(s)})^\circ
    =
    \left\{z\in\RR^p:\|z_{T_s(z)}\|_{K^\circ}\le 1\right\}.
\end{align*}
If, in addition,
\begin{align*}
    rB_2^p\subseteq K\subseteq RB_2^p,
\end{align*}
then
\begin{align*}
    \frac1R B_2^p
    \subseteq
    (K^{(s)})^\circ
    \subseteq
    \frac1r\sqrt{\frac{p}{s}}B_2^p.
\end{align*}
Thus $(K^{(s)})^\circ$ is well-balanced, with Euclidean aspect ratio at most $(R/r)\sqrt{p/s}$.
\end{lemma}

\begin{proof}
Since support functions are unchanged by taking convex hulls,
\begin{align*}
    h_{K^{(s)}}(z)
    =
    \sup_{v\in S_s}\langle z,v\rangle
    =
    \max_{|T|\le s}\ \sup_{\substack{v\in K\\ \operatorname{supp}(v)\subseteq T}}
    \langle z,v\rangle .
\end{align*}
Because $K$ is sign-symmetric and permutation-invariant convex body, coordinate projections preserve membership in $K$.  Hence, for each fixed $T$,
\begin{align*}
    \sup_{\substack{v\in K\\ \operatorname{supp}(v)\subseteq T}}
    \langle z,v\rangle
    =
    \sup_{v\in K}\langle z_T,v\rangle
    =
    h_K(z_T)
    =
    \|z_T\|_{K^\circ}.
\end{align*}
Therefore
\begin{align*}
    h_{K^{(s)}}(z)=\max_{|T|\le s}\|z_T\|_{K^\circ}.
\end{align*}
The norm $\|\cdot\|_{K^\circ}$ is also sign-invariant and permutation-invariant, so the maximum over all $T$ with $|T|\le s$ is attained by taking the indices of the $s$ largest coordinates of $z$ in absolute value.  This proves the support-function identity and the formula for the polar.

It remains to check the Euclidean sandwich.  Polarity of
$rB_2^p\subseteq K\subseteq RB_2^p$ gives
\begin{align*}
    \frac1R B_2^p\subseteq K^\circ\subseteq \frac1r B_2^p,
\end{align*}
or equivalently
\begin{align*}
    r\|w\|_2\le \|w\|_{K^\circ}\le R\|w\|_2
    \qquad \text{for all }w\in\RR^p.
\end{align*}
If $\|z\|_2\le 1/R$, then
\begin{align*}
    h_{K^{(s)}}(z)=\|z_{T_s(z)}\|_{K^\circ}
    \le R\|z_{T_s(z)}\|_2
    \le R\|z\|_2
    \le 1,
\end{align*}
so $z\in (K^{(s)})^\circ$.  Conversely, if $z\in (K^{(s)})^\circ$, then
\begin{align*}
    1\ge h_{K^{(s)}}(z)=\|z_{T_s(z)}\|_{K^\circ}
    \ge r\|z_{T_s(z)}\|_2.
\end{align*}
Since $T_s(z)$ contains the $s$ largest coordinates of $z$,
\begin{align*}
    \|z_{T_s(z)}\|_2^2\ge \frac{s}{p}\|z\|_2^2.
\end{align*}
Hence $\|z\|_2\le r^{-1}\sqrt{p/s}$, proving the upper inclusion.
\end{proof}

This identity shows that an exact membership oracle for $K^\circ$ immediately yields an exact membership oracle for $(K^{(s)})^\circ$. In our setting, however, we start only with a membership oracle for $K$ (which is easy to obtain since we are allowed to evaluate the gauge of $K$), which, as shown below, yields only a weak membership oracle for $K^\circ$. We therefore need the following stability lemma, which shows that weak membership also passes from $K^\circ$ to $(K^{(s)})^\circ$, up to a polynomial loss in precision. The proof is deferred to the appendix.

\begin{lemma}[Weak membership for $(K^{(s)})^\circ$]
\label{lem:weak-membership-Ks-polar}
Assume that $K\subseteq \mathbb R^p$ satisfies
\begin{align*}
    rB_2^p\subseteq K\subseteq RB_2^p
\end{align*}
and is convex, sign-symmetric, and permutation-invariant.

Suppose that we have a weak membership oracle for $K^\circ$. Then we have a
weak membership oracle for $(K^{(s)})^\circ$.

More precisely, on input $z\in\mathbb R^p$ and $\varepsilon>0$, query the weak membership oracle for $K^\circ$ with input $z_{T_s(z)}$ and tolerance  $\delta
    =
    \frac{r}{R}\sqrt{\frac{s}{p}}\,\varepsilon$. This gives a valid weak membership oracle for $(K^{(s)})^\circ$.
\end{lemma}

Lemma~\ref{lem:weak-membership-Ks-polar} shows that a weak membership
oracle for $K^\circ$ still gives a weak membership oracle for
$(K^{(s)})^\circ$, with only the polynomial loss in precision
\begin{align*}
    \delta
    =
    \frac{r}{R}\sqrt{\frac{s}{p}}\,\varepsilon .
\end{align*}
On the other hand, applying Lemma~\ref{lem:polar-weak-membership}, given below, with
$C=K^\circ$ shows that a weak membership oracle for $K$ gives a weak
membership oracle for $K^\circ$, since
\begin{align*}
    \frac1R B_2^p\subseteq K^\circ\subseteq \frac1r B_2^p .
\end{align*}
Combining these two reductions, a weak membership oracle for $K$ gives a
weak membership oracle for $(K^{(s)})^\circ$. Finally, since
$(K^{(s)})^\circ$ is well-balanced by Lemma~\ref{lem:Ks-polar}, another
application of Lemma~\ref{lem:polar-weak-membership} below, now with
$C=K^{(s)}$ (which is also well balanced), gives a weak membership oracle for $K^{(s)}$.

\begin{lemma}[From polar weak membership to primal weak membership]
\label{lem:polar-weak-membership}
Let $C\subset \mathbb R^p$ be a closed, convex, centrally symmetric body with
$0\in \operatorname{int}(C)$. Assume that
\begin{align*}
    r B_2^p \subseteq C \subseteq R B_2^p
\end{align*}
for known $0<r\le R<\infty$. Suppose that we have a weak membership oracle
for the polar body
\begin{align*}
    C^\circ
    :=
    \{y\in\mathbb R^p:\langle y,x\rangle\le 1
    \text{ for all }x\in C\}.
\end{align*}
That is, given $y\in\mathbb R^p$ and $\delta>0$, the oracle returns
$\textnormal{YES}$ if $y\in C^\circ$, returns $\textnormal{NO}$ if
$\operatorname{dist}(y,C^\circ)>\delta$, and may return arbitrary output
otherwise.

Then there is a polynomial-time oracle procedure giving a weak membership
oracle for $C$. More precisely, given $x\in\mathbb R^p$ and
$\varepsilon>0$, the procedure returns
\begin{align*}
    \textnormal{YES} \quad \text{if } x\in C,
\end{align*}
and
\begin{align*}
    \textnormal{NO} \quad \text{if } \operatorname{dist}(x,C)>\varepsilon,
\end{align*}
with arbitrary output on the boundary band
\begin{align*}
    \{x:\operatorname{dist}(x,C)\le \varepsilon\}.
\end{align*}
\end{lemma}

Thus by Theorem 2.5.9 \cite{dadush2012integer} the inner optimization in \eqref{hd:estimator} can be approximated in time polynomial in $p$, $\log(R/r)$, and the requested accuracy. This is because the function $v \mapsto \|Y - X v\|_2$ is Lipschitz with high probability. For simplicity we will assume henceforth that we can optimize the problem exactly, but it is easy to see the same results continue to hold if we set the precision to a sufficiently small number. 

Finally we state two lemmas, establishing that the Gaussian width of the set $K^{(s)}$ can be computed up to constant factors. This will be very useful when we formally state the algorithm and penalty in the next section. The proofs are deferred to the appendix.

\begin{lemma}[Order-statistic bound for $w(K^{(s)})$]
\label{lem:width-order-stat}
Let $K\subseteq \mathbb R^p$ be centrally symmetric, convex,
sign-invariant and permutation-invariant.  Let $g\sim N(0,I_p)$, and let
$g_1^*\ge\cdots\ge g_p^*$ denote the decreasing rearrangement of
$(|g_1|,\ldots,|g_p|)$.  Define
\begin{align*}
    \gamma^{(s)}
    :=
    \left(
    \sqrt{\log(ep)},\sqrt{\log(ep/2)},\ldots,
    \sqrt{\log(ep/s)},0,\ldots,0
    \right).
\end{align*}
Then
\begin{align*}
    w(K^{(s)})
    =
    \mathbb E\left\|
    (g_1^*,\ldots,g_s^*,0,\ldots,0)
    \right\|_{K^\circ}
\end{align*}
and
\begin{align*}
    w(K^{(s)})
    \le
    C\|\gamma^{(s)}\|_{K^\circ}
\end{align*}
\end{lemma}

\begin{lemma}[Matching lower bound for the order-statistic width]\label{lem:matching-lower-bound-gaussian-width}
Under the same assumptions as in Lemma \ref{lem:width-order-stat} there is a universal constant $c>0$ such that
\begin{align*}
    w(K^{(s)})
    \ge
    c\|\gamma^{(s)}\|_{K^\circ}.
\end{align*}
Consequently,
\begin{align*}
    w(K^{(s)}) \asymp \|\gamma^{(s)}\|_{K^\circ},
\end{align*}
with universal constants.
\end{lemma}

\subsection{Main result and proof}

We now state and prove a version of the oracle inequality in which the
complexity of $K^{(s)}$ is measured directly by its Gaussian width.

For $1\le s\le p$, define
\begin{align*}
    W_s:=w(\Sigma^{1/2}K^{(s)})
    =
    \EE\sup_{v\in K^{(s)}}\langle g,\Sigma^{1/2}v\rangle,
    \qquad g\sim N(0,I_p),
\end{align*}
and
\begin{align*}
    D_s:=\diam_2(\Sigma^{1/2}K^{(s)}).
\end{align*}
Thus $D_s\le \sqrt{\lambda_+}\,d_s$, where $d_s := \operatorname{diam}_2(K^{(s)})$. By Sudakov-Fernique we further have $\sqrt{\lambda_{\min}(\Sigma)}w(K^{(s)})\, \leq W_s = w(\Sigma^{1/2}K^{(s)}) \leq \sqrt{\lambda_{\max}(\Sigma)}\,w(K^{(s)}),$ and thus $W_s \asymp w_s := w(K^{(s)})$. Finally we mention that it also holds that $d_s \lesssim w_s$ \cite{vershynin2018high}. Let
\begin{align*}
    \ell_s:=c\log(e+s),
\end{align*}
for some large $c > 6$, and define the width complexity
\begin{align*}
    \operatorname{pen}(s)
    := \sigma\bigg(\frac{ w_s \sqrt{\ell_s}}{\sqrt{N}} + \frac{ w_s \ell_s}{N}\bigg) + \frac{ w_s^2\ell_s}{N}
\end{align*}
The small factors $\frac{ \sigma w_s \ell_s}{N}, \frac{ w_s^2\ell_s}{N}$ are usually negligible, but keeping them makes the high-probability multiplier bound completely clean.

The Gaussian-width penalized estimator is
\begin{align}\label{hd:estimator-width}
    (\widehat s,\widehat\beta)
    \in
    \argmin_{\substack{1\le s\le p, \\\nu\in K^{(s)}}}
    \left\{
        \frac1N\|Y-X\nu\|_2^2
        +
        \lambda
        \operatorname{pen}(s)
    \right\}.
\end{align}
More generally, one may replace $w_s$ by any computable upper bound
$\widetilde w_s\ge w_s $ in the definition of $\operatorname{pen}(s)$, provided that
\begin{align*}
    \widetilde w_{s+t}
    \le
    C\{\widetilde w_s+\widetilde w_t\},
    \qquad 1\le s,t,\ s+t\le p,
\end{align*}
For simplicity we state the theorem using the exact $w_s$, while noting that one may compute $w_s$ up to absolute constants by Lemmas \ref{lem:width-order-stat} and \ref{lem:matching-lower-bound-gaussian-width}. Thus clearly this computable version satisfies the above requirements as per Lemma \ref{lem:width-complexity-subadditive} $w_{s + t} \leq w_s + w_t$ in the admissible range.

\begin{theorem}[Oracle inequality with Gaussian width]
\label{thm:width-oracle}
Assume that $K\subseteq\RR^p$ is centrally symmetric, convex,
sign-invariant and permutation-invariant. Assume that the rows of $X$
are independent, mean-zero, $L$-sub-Gaussian with covariance $\Sigma$
satisfying
\begin{align*}
    \lambda_- I_p\preceq \Sigma\preceq \lambda_+ I_p,
\end{align*}
and that $\xi_1,\ldots,\xi_N$ are independent, mean-zero, sub-Gaussian
with $\|\xi_i\|_{\psi_2}\le \sigma$, independent of $X$.
Let $\widehat\beta$ be defined by \eqref{hd:estimator-width}.  If
$\lambda\ge \lambda_0$, where $\lambda_0$ depends only on
$L,\lambda_-,\lambda_+$, then
\begin{align*}
    \EE\|\widehat\beta-\beta\|_2^2
    \le
    C
    \inf_{1\le t\le p}
    \inf_{x\in K^{(t)}}
    \left\{
        \|x-\beta\|_2^2
        +
        (\lambda + 1)\operatorname{pen}(t)
    \right\}
\end{align*}
\end{theorem}

\subsubsection{Uniform empirical norm control}

We use the matrix deviation inequality directly in terms of Gaussian width.

\begin{lemma}[Width-based empirical norm bound]
\label{lem:width-empirical}
Under the assumptions of Theorem~\ref{thm:width-oracle}, 
\begin{align*}
   \EE \sup_{s \in [p]} \bigg[ \sup_{v\in 2K^{(s)}}
    \left|
        \frac{\|Xv\|_2}{\sqrt N}
        -
        \|\Sigma^{1/2}v\|_2
    \right|^2 - C^2
    \left\{
        \frac{w_s}{\sqrt N}
        +
        d_s\sqrt{\frac{\ell_s}{N}}
    \right\}^2\bigg]_+
    \le
    C
        d_1^2 \frac{1}{N} \lesssim \frac{w_1^2}{N}.
\end{align*}
We will abbreviate $q_s := C^2
    \left\{
        \frac{w_s}{\sqrt N}
        +
        d_s\sqrt{\frac{\ell_s}{N}}
    \right\}^2 \lesssim w_s^2 \frac{\ell_s}{N}$,  where we used the well known bound $d_s \lesssim w_s$.
\end{lemma}

\begin{proof}

Let $Z_i=\Sigma^{-1/2}X_i$.  Then $Z_i$ is isotropic and $L$-sub-Gaussian. Applying Remark 9.1.4 of \cite{vershynin2018high} to the set $\Sigma^{1/2}K^{(s)}$ gives
\begin{align*}
    \sup_{v\in K^{(s)}}
    \left|
        \frac{\|Xv\|_2}{\sqrt N}-\|\Sigma^{1/2}v\|_2
    \right|
    \le
    C L^2\frac{w(\Sigma^{1/2}K^{(s)})+\sqrt{u}\operatorname{rad}(\Sigma^{1/2}K^{(s)})}{\sqrt N},
\end{align*}
Observe that we have
\begin{align*}
    w(\Sigma^{1/2}K^{(s)})
    &\le
    \sqrt{\lambda_{\max}(\Sigma)}\,w(K^{(s)}),
    \\
    \diam(\Sigma^{1/2}K^{(s)})
    &\le
    \sqrt{\lambda_{\max}(\Sigma)}\,d_s.
\end{align*}
where the first bound follows by Sudakov-Fernique's inequality. Thus since we are assuming universal bounds on the eigenvalues of $\Sigma$ we obtain, with probability at least $1-Ce^{-x}$,
\begin{align*}
    \sup_{v\in 2K^{(s)}}
    \left|
        \frac{\|Xv\|_2}{\sqrt N}
        -
        \|\Sigma^{1/2}v\|_2
    \right|
    \le
    C
    \left\{
        \frac{w_s}{\sqrt N}
        +
        d_s\sqrt{\frac{x}{N}}
    \right\}.
\end{align*}
Set $x=\ell_s+u$. Consider the variable

\begin{align*}
    Z = \bigg(\bigg(\sup_{v\in 2K^{(s)}}
    \left|
        \frac{\|Xv\|_2}{\sqrt N}
        -
        \|\Sigma^{1/2}v\|_2
    \right|^2 - A \bigg(\frac{w_s}{\sqrt N}
        +
        d_s\sqrt{\frac{\ell_s}{N}}\bigg)^2\bigg),
\end{align*}
for a large enough constant $A$. Let $Z_+ = (Z)_+$. Then upon using the formula $\EE Z_+ = \int_{0}^\infty \PP(Z_+ > t) dt =\int_{0}^\infty \PP(Z > t) dt
$ it is easy to see that
\begin{align*}
    \EE Z_+ \lesssim e^{-\ell_s} d_s^2/N.
\end{align*}

Next as a lemma in the appendix shows $d_s^2 \lesssim s d_1^2$ (see Lemma \ref{lem:diameter-penalty-quasi-subadditive}) and hence upon summing this we conclude that 
\begin{align*}
    \EE \sup_{s \geq 1} Z_+ \leq \sum_{s = 1}^p e^{-\ell_s} d_s^2/N\lesssim d_1^2/N.
\end{align*}

\end{proof}

\subsubsection{Noise multiplier bound}

We next present a noise multiplier bound with a direct Gaussian-width multiplier bound.

\begin{lemma}[Width-based multiplier bound]
\label{lem:width-multiplier}
Under the assumptions of Theorem~\ref{thm:width-oracle}, 
\begin{align*}
    \MoveEqLeft\EE \sup_{s\geq 1}\bigg( \sup_{v\in 2K^{(s)}}
    \frac{2}{N}|\langle Xv,\xi\rangle| - A\sigma \left[
        \frac{w_s}{\sqrt N}
        \left(1+\sqrt{\frac{\ell_s}{N}}\right)
        +
        d_s
        \left(
            \sqrt{\frac{\ell_s}{N}}
            +
            \frac{\ell_s}{N}
        \right)
    \right]\bigg)_+ \\
    & \lesssim \sigma\bigg(\frac{w_1}{ \sqrt{N}}
        +
        d_1
        \left(
            \sqrt{\frac{1}{N}}
            +
            \frac{1}{N}\bigg)
        \right)\\
        & \lesssim \sigma\frac{w_1}{ \sqrt{N}}
\end{align*}

We will abbreviate $z_s :=  A\sigma \left[
        \frac{w_s}{\sqrt N}
        \left(1+\sqrt{\frac{\ell_s}{N}}\right)
        +
        d_s
        \left(
            \sqrt{\frac{\ell_s}{N}}
            +
            \frac{\ell_s}{N}
        \right)
    \right] \lesssim \sigma \bigg( w_s \sqrt{\frac{\ell_s}{N}} + w_s \frac{\ell_s}{N}\bigg)$, where we used the well known bound $d_s \lesssim w_s$.
\end{lemma}

\begin{proof}
Fix $s$ and condition on $\xi=(\xi_1,\ldots,\xi_N)$.  For
$v\in 2K^{(s)}$, define
\begin{align*}
    Z_v:=\frac1N\sum_{i=1}^N \xi_i\langle X_i,v\rangle .
\end{align*}
For $v,w\in 2K^{(s)}$, by independence and the sub-Gaussian assumption on
the rows of $X$,
\begin{align*}
    \|Z_v-Z_w\|_{\psi_2\mid \xi}
    \le
    C L\frac{\|\xi\|_2}{N}
    \|\Sigma^{1/2}(v-w)\|_2 .
\end{align*}
Hence, conditionally on $\xi$, $\{Z_v:v\in 2K^{(s)}\}$ is a
sub-Gaussian process with respect to the metric
\begin{align*}
    d_\xi(v,w)
    =
    C L\frac{\|\xi\|_2}{N}
    \|\Sigma^{1/2}(v-w)\|_2 .
\end{align*}
By the generic chaining bound for sub-Gaussian processes
\cite[Theorem 8.5.2 and Remark 8.5.4]{vershynin2018high},
together with the majorizing-measure comparison between $\gamma_2$ and
Gaussian width, with conditional probability at least $1-Ce^{-x}$,
\begin{align*}
    \sup_{v\in 2K^{(s)}} |Z_v|
    \le
    C L\frac{\|\xi\|_2}{N}
    \left\{
        w_s+d_s\sqrt{x}
    \right\}.
\end{align*}
Equivalently,
\begin{align*}
    \sup_{v\in 2K^{(s)}}
    \frac{1}{N}|\langle Xv,\xi\rangle|
    \le
    C L\frac{\|\xi\|_2}{N}
    \left\{
        w_s+d_s\sqrt{x}
    \right\}.
\end{align*}

By concentration of the Euclidean norm of a vector with independent
sub-Gaussian entries,
\begin{align*}
    \|\xi\|_2
    \le
    C\sigma(\sqrt N+\sqrt{x})
\end{align*}
with probability at least $1-Ce^{-x}$. Therefore, with probability at
least $1-Ce^{-x}$,
\begin{align*}
    \sup_{v\in 2K^{(s)}}
    \frac{1}{N}|\langle Xv,\xi\rangle|
    \le
    C\sigma
    \frac{\sqrt N+\sqrt{x}}{N}
    \left\{
        w_s+d_s\sqrt{x}
    \right\}.
\end{align*}
Expanding the product gives
\begin{align*}
    \sup_{v\in 2K^{(s)}}
    \frac{1}{N}|\langle Xv,\xi\rangle|
    \le
    C\sigma
    \left[
        \frac{w_s}{\sqrt N}
        +
        \frac{w_s\sqrt{x}}{N}
        +
        d_s\sqrt{\frac{x}{N}}
        +
        d_s\frac{x}{N}
    \right].
\end{align*}
Finally, set $x=\ell_s+u$ to see the first identity. 

Similarly to the proof of Lemma \ref{lem:width-empirical} we now have that 
\begin{align*}
\MoveEqLeft \EE \sup_{v\in 2K^{(s)}}
     \bigg(\frac{2}{N}|\langle Xv,\xi\rangle| - A\sigma \left[
        \frac{w_s}{\sqrt N}
        \left(1+\sqrt{\frac{\ell_s}{N}}\right)
        +
        d_s
        \left(
            \sqrt{\frac{\ell_s}{N}}
            +
            \frac{\ell_s}{N}
        \right)
    \right] \bigg)_+ \\
    &\lesssim e^{-\ell_s} \sigma \bigg(\frac{w_s}{\sqrt{N}} + d_s \bigg(\frac{1}{\sqrt{N}} + \frac{1}{N}\bigg)\bigg).
\end{align*}
Thus the proof is completed by telescoping as before upon noting that $w_s \leq sw_1$ and $d_s \lesssim \sqrt{s} d_1 \lesssim s d_1$ as shown in Lemma \ref{lem:width-complexity-subadditive} and Lemma \ref{lem:diameter-penalty-quasi-subadditive} in the appendix.
\end{proof}

\subsubsection{Basic inequality and oracle inequality}

Fix $t\in[p]$ and $x\in K^{(t)}$.  Define
\begin{align*}
    a:=x-\beta,
    \qquad
    h:=\widehat\beta-x.
\end{align*}
Then
\begin{align*}
    \widehat\beta-\beta=h+a.
\end{align*}
Since
\begin{align*}
    \widehat\beta\in K^{(\widehat s)},
    \qquad
    x\in K^{(t)},
\end{align*}
we have
\begin{align*}
    h=\widehat\beta-x
    \in
    K^{(\widehat s)}-K^{(t)}
    \subseteq
    2K^{(\widehat s+t)}.
\end{align*}
Set
\begin{align*}
    m:=(\widehat s+t)\wedge p.
\end{align*}

By the definition of $(\widehat s,\widehat\beta)$,
\begin{align*}
    \frac1N\|Y-X\widehat\beta\|_2^2
    +
    \lambda \operatorname{pen}(\widehat s)
    \le
    \frac1N\|Y-Xx\|_2^2
    +
    \lambda \operatorname{pen}(t).
\end{align*}
Using $Y=X\beta+\xi$, expanding the squares, and applying Cauchy's
inequality gives
\begin{align*}
    \|\Sigma^{1/2}h\|_2^2/2
    &\le\|\Sigma^{1/2}h\|_2^2/2-\frac1N\|Xh\|_2^2 + 
    \frac4N\|Xa\|_2^2
    +
    \frac4N|\langle \xi,Xh\rangle|
    +
    2\lambda \cdot \operatorname{pen}(t)
    -
    2\lambda \cdot \operatorname{pen}(\widehat s)\\
    & \leq (\|\Sigma^{1/2}h\|_2-\frac{1}{\sqrt{N}}\|Xh\|_2)^2 + 
    \frac4N\|Xa\|_2^2
    +
    \frac4N|\langle \xi,Xh\rangle|
    +
    2\lambda \cdot \operatorname{pen}(t)
    -
    2\lambda \cdot \operatorname{pen}(\widehat s)
    \label{eq:width-basic-prediction}\\
    & \leq ((\|\Sigma^{1/2}h\|_2-\frac{1}{\sqrt{N}}\|Xh\|_2)^2 - q_m)_+ + q_m  \\
    &+ 
    \frac4N\|Xa\|_2^2
    +
    4(\frac1N|\langle \xi,Xh\rangle| - C z_m)_+ + 4Cz_m
    +
    2\lambda \cdot \operatorname{pen}(t) -
    2\lambda \cdot \operatorname{pen}(\widehat s)\\
    & \leq ((\|\Sigma^{1/2}h\|_2-\frac{1}{\sqrt{N}}\|Xh\|_2)^2 - q_m)_+ + q_m \\
    &+ 
    \frac4N\|Xa\|_2^2
    +
    4(\frac1N|\langle \xi,Xh\rangle| - C z_m)_+ + 4Cz_m
    +
    2\lambda \cdot \operatorname{pen}(t) -
    2\lambda \cdot \operatorname{pen}(\widehat s).
\end{align*}
Thus
\begin{align*}
    \|\Sigma^{1/2}h\|_2^2/2 & \lesssim ((\|\Sigma^{1/2}h\|_2-\frac{1}{\sqrt{N}}\|Xh\|_2)^2 - q_m)_+ + \sigma\bigg(\frac{w_m \sqrt{\ell_m}}{\sqrt{N}} + \frac{w_m {\ell_m}}{{N}}\bigg) + \frac{w_m^2 \ell_m}{N}  \\
    &+ 
    \frac1N\|Xa\|_2^2
    +
    2(\frac1N|\langle \xi,Xh\rangle| - C z_m)_+ 
    +
    \lambda \cdot \operatorname{pen}(t) -
    \lambda \cdot \operatorname{pen}(\widehat s),
\end{align*}
where we used the definitions of $q_m$ and $z_m$. Next by Lemma \ref{lem:width-complexity-subadditive} (below) using $w_{m} \leq w_{\widehat s} + w_t$, (which implies) $w^2_{m} \lesssim w^2_{\widehat s} + w^2_t$ and also the fact that $\ell_{m} \lesssim \ell_{\widehat s} + \ell_{t}$ (and using a maximal reasoning) we conclude:
\begin{align}
     \sigma\bigg(\frac{w_m \sqrt{\ell_m}}{\sqrt{N}} + \frac{w_m {\ell_m}}{{N}}\bigg) + \frac{w_m^2 \ell_m}{N}  \lesssim \sigma\bigg(\frac{w_{\widehat s} \sqrt{\ell_{\widehat s}}}{\sqrt{N}}+ \frac{w_{\widehat s} {\ell_{\widehat s}}}{{N}} \bigg)+ \frac{w_{\widehat s}^2\ell_{\widehat s}}{N} + \sigma\bigg(\frac{w_t \sqrt{\ell_t}}{\sqrt{N}} + \frac{w_t {\ell_t}}{{N}}\bigg) + \frac{w_t^2 \ell_t}{N}.
\end{align}
It follows that when $\lambda$ is sufficiently large $\lambda \cdot \operatorname{pen}(\widehat s)$ can absorb the term $\sigma\bigg(\frac{w_{\widehat s} \sqrt{\ell_{\widehat s}}}{\sqrt{N}} + \frac{w_{\widehat s} {\ell_{\widehat s}}}{{N}}\bigg) + \frac{w_{\widehat s}^2\ell_{\widehat s}}{N}$, and in addition the term $\sigma\bigg(\frac{w_t \sqrt{\ell_t}}{\sqrt{N}} + \frac{w_t {\ell_t}}{{N}} \bigg)+ \frac{w_t^2 \ell_t}{N}$ can be absorbed in $\lambda\operatorname{pen}(t)$.

Taking the supremum over $h\in 2K^{(m)}$ and then over $m\ge1$, we obtain
\begin{align*}
    \lambda_-/2 \|h\|_2^2 \leq \|\Sigma^{1/2}h\|_2^2/2 & \lesssim \sup_{m \geq 1}\sup_{h \in 2 K^{(m)}}((\|\Sigma^{1/2}h\|_2-\frac{1}{\sqrt{N}}\|Xh\|_2)^2 - q_m)_+  \\
    &+ 
    C\frac1N\|Xa\|_2^2
    +
    \sup_{m \geq 1}\sup_{h \in 2 K^{(m)}} (C\frac1N|\langle \xi,Xh\rangle| - C z_m)_+ 
    +
    C\lambda \cdot \operatorname{pen}(t).
\end{align*}
Note that $\|\hat\beta - \beta\|_2^2 = \|h + a\|_2^2 \leq 2 \|h\|_2^2 + 2 \|a\|_2^2$. Taking expectations, using
\begin{align*}
    \EE\frac1N\|Xa\|_2^2
    =
    a^\top\Sigma a
    \le
    \lambda_+\|a\|_2^2,
\end{align*}
we obtain
\begin{align*}
    \EE\|\widehat\beta-\beta\|_2^2
    \le
    C\|x-\beta\|_2^2
    +
    C\lambda \operatorname{pen}(t)
    +
    C\sigma\frac{w_1}{\sqrt{N}}
     + C\frac{w_1^2}{N}
    \le
    C\|x-\beta\|_2^2
    +
    (C\lambda + C) \operatorname{pen}(t)
\end{align*}
Since $t$ and $x\in K^{(t)}$ were arbitrary, taking the infimum over
$t$ and $x\in K^{(t)}$ proves Theorem~\ref{thm:width-oracle}.

\begin{lemma}[Quasi-subadditivity of the width complexity]
\label{lem:width-complexity-subadditive}
For $s,t\ge1$, with the convention $K^{(r)}=K^{(p)}$ for $r\ge p$,
\begin{align*}
    w_{s+t}
    \le
    \{w_s+ w_t\}.
\end{align*}
Consequently,
\begin{align*}
    w_{s+t}^2
    \le
    C\{w_s^2+w_t^2\}.
\end{align*}
Furthermore, 
\begin{align*}
    \ell_{s + t}=c\log(e+s+t)
    \le
    C\{\ell_s+\ell_t\}.
\end{align*}
\end{lemma}

\begin{proof}
We first claim that
\begin{align*}
    K^{(s+t)}
    \subseteq
    K^{(s)}+K^{(t)}.
\end{align*}
Indeed, let $v\in K$ have support of size at most $s+t$.  Split its
support into two sets $A$ and $B$ with $|A|\le s$ and $|B|\le t$.
Since $K$ is sign-invariant and permutation-invariant, coordinate
projections preserve membership in $K$. Hence $v_A\in K^{(s)}$ and
$v_B\in K^{(t)}$, and $v=v_A+v_B$. Taking convex hulls gives the
claim.

Therefore
\begin{align*}
    \Sigma^{1/2} K^{(s+t)}
    \subseteq
    \Sigma^{1/2}K^{(s)}+\Sigma^{1/2}K^{(t)}.
\end{align*}
Gaussian width is subadditive under Minkowski sums, so
\begin{align*}
    W_{s+t}
    \le
    W_s+W_t,
\end{align*}
and also
\begin{align*}
    w_{s+t}
    \le
    w_s+w_t,
\end{align*}

Finally it is clear that,
\begin{align*}
    \ell_{s + t}=c\log(e+s+t)
    \le
    C\{\ell_s+\ell_t\}.
\end{align*}
This completes the proof of the result.
\end{proof}

\section{Minimax Optimality}\label{minimax:sec}

In this section we argue that our procedure achieves minimax optimality over the set $K^{(s)}$ for some $s \in [p]$, assuming that the true $\beta \in K^{(s)}$, that the sample size is large enough.

We begin by a simple lemma relating the entropy numbers to a hard thresholding functional.

\begin{lemma}[Entropy lower bound from the Edmunds--Netrusov threshold functional]
\label{lem:EN-threshold-implies-tail}
Let $L\subset\mathbb R^p$ be a convex, centrally symmetric,
sign-symmetric, and permutation-invariant body.  For $1\le m\le p$, define
\begin{align*}
    a_m(L)
    :=
    \sup_{x\in L}
    \left(\sum_{i=m+1}^p (x_i^\ast)^2\right)^{1/2},
\end{align*}
where $x^\ast=(x_1^\ast,\ldots,x_p^\ast)$ denotes the non-increasing
rearrangement of $(|x_1|,\ldots,|x_p|)$.  Define also the thresholding
functional
\begin{align*}
    b_m(L)
    :=
    \sup_{x\in L}
    \left(
        \sum_{i=1}^p \min\{x_i^\ast,x_m^\ast\}^2
    \right)^{1/2}.
\end{align*}
The Edmunds--Netrusov \cite{edmunds1998entropy} entropy estimate gives, for some universal
constants $c_0,c_1>0$,
\begin{align*}
    e_{\lceil c_1 m\log(ep/m)\rceil}(L)
    \ge
    c_0\, b_m(L),
\end{align*}
so long as $c_1 m\log(ep/m) \leq p/2$.
Under the same condition we have
\begin{align*}
    e_{\lceil c_1 m\log(ep/m)\rceil}(L)
    \gtrsim
    a_m(L).
\end{align*}
\end{lemma}

\begin{proof}
Fix $x\in L$.  Since $x^\ast$ is non-increasing, for every $i>m$ we have
\begin{align*}
    x_i^\ast\le x_m^\ast.
\end{align*}
Therefore
\begin{align*}
    \min\{x_i^\ast,x_m^\ast\}=x_i^\ast,
    \qquad i=m+1,\ldots,p.
\end{align*}
Hence
\begin{align*}
    \sum_{i=1}^p \min\{x_i^\ast,x_m^\ast\}^2
    \ge
    \sum_{i=m+1}^p \min\{x_i^\ast,x_m^\ast\}^2
    =
    \sum_{i=m+1}^p (x_i^\ast)^2.
\end{align*}
Taking square roots gives
\begin{align*}
    \left(
        \sum_{i=1}^p \min\{x_i^\ast,x_m^\ast\}^2
    \right)^{1/2}
    \ge
    \left(
        \sum_{i=m+1}^p (x_i^\ast)^2
    \right)^{1/2}.
\end{align*}
Taking the supremum over $x\in L$, we obtain
\begin{align*}
    b_m(L)\ge a_m(L).
\end{align*}
Combining this deterministic comparison with the Edmunds--Netrusov estimate,
\begin{align*}
    e_{\lceil c_1 m\log(ep/m)\rceil}(L)
    \ge
    c_0 b_m(L)
    \ge
    c_0 a_m(L).
\end{align*}
This proves the claim.
\end{proof}

We continue with stating a Lemma which contains the crux of the minimax proof. It turns out that Sudakov minoration is tight for a regime of values over the set $K^{(s)}$.

\begin{lemma}[Sudakov tightness on a range of sparse entropy scales]
\label{lem:sudakov-tight-sparse-convexification}
Let $K\subset \mathbb R^p$ be a convex, centrally symmetric, sign-symmetric,
and permutation-invariant body, and for $1\le s\le p$ define
\begin{align*}
    K^{(s)}
    :=
    \operatorname{conv}\{x\in K:\|x\|_0\le s\}.
\end{align*}
Set $L:=K^{(s)}$.  For $1\le q\le p$, define the dual fundamental
function
\begin{align*}
    \phi_K(q)
    :=
    \left\|\sum_{i=1}^q e_i\right\|_{K^\circ}
    =
    \sup_{x\in K}\sum_{i=1}^q x_i .
\end{align*}
By Lemma \ref{lem:width-order-stat} we have
\begin{align*}
    w(L)
    \le
    A\,\phi_K(s)\sqrt{\log e p}
\end{align*}
for some constant $A\ge 1$, where
\begin{align*}
    w(L):=\mathbb E\sup_{x\in L}\langle g,x\rangle,
    \qquad g\sim N(0,I_p).
\end{align*}

Then for every $m$ satisfying
\begin{align*}
    s\le m\le p/c
\end{align*}
for some sufficiently large absolute constant $c$, if
\begin{align*}
    k_m:=\left\lceil c_1 m\log(ep/m)\right\rceil \leq p/2,
\end{align*}
then
\begin{align*}
    \sqrt{k_m}\,e_{k_m}(L) 
    \gtrsim_{A,c_0,c_1,c, }
    \frac{w(L)}{\sqrt{\log e p}}.
\end{align*}
On the other hand, Sudakov's minoration gives
\begin{align*}
    \sqrt{k}\,e_k(L)\lesssim w(L)
    \qquad\text{for all } k\ge 1.
\end{align*}
Consequently, for all such $m$,
\begin{align*}
    \sqrt{k_m}\,e_{k_m}(K^{(s)})
    \asymp_{A,c_0,c_1,c, \sqrt{\log e p}}
    w(K^{(s)}).
\end{align*}
\end{lemma}

\begin{proof} We start by arguing that by Lemma \ref{lem:width-order-stat} we have
\begin{align*}
    w(L)
    \le
    A\,\phi_K(s)\sqrt{\log e p}.
\end{align*}
Indeed since the dual norm $\|\cdot\|_{K^\circ}$ is symmetric and the vector $(\sqrt{\log ep},\sqrt{\log ep}, \ldots, \sqrt{\log ep}, 0, 0,\ldots,0)$ dominates coordinatewise the vector $(\sqrt{\log ep/1},\sqrt{\log ep/2}, \ldots, \sqrt{\log (ep/s)}, 0, 0,\ldots,0)$ the claim follows by Lemma \ref{lem:width-order-stat}.

We next show that sparse convexification contains flat vectors at all larger
scales.  By the definition of the polar norm,
\begin{align*}
    \phi_K(s)
    =
    \sup_{x\in K}\sum_{i=1}^{s}x_i .
\end{align*}
Since $K$ is sign-symmetric and convex, coordinate projections of points in
$K$ also belong to $K$.  Hence the supremum above may be taken over
$s$-sparse vectors.  Thus, up to an arbitrarily small approximation
error, there exists $y\in K$ with $\operatorname{supp}(y)\subseteq
\{1,\ldots,s\}$ such that
\begin{align*}
    \sum_{i=1}^{s}y_i
    \ge
    \frac12\,\phi_K(s).
\end{align*}
Changing signs if necessary, we may assume $y_i\ge 0$.  Since $y$ is $s$-sparse and $K$ is permutation-invariant, every coordinate permutation of $y$ belongs to $K^{(s)}$.

Fix $m\ge s$ with $2m\le p$, and let $R\subset\{1,\ldots,p\}$ have
cardinality $|R|=2m$.  Average all permutations of $y$ supported inside
$R$.  By convexity and permutation-invariance of $K^{(s)}$, the averaged
vector belongs to $K^{(s)}$.  This average is the flat vector
\begin{align*}
    z
    =
    \frac{1}{2m}\left(\sum_{i=1}^{s}y_i\right)\mathbf 1_R .
\end{align*}
Therefore
\begin{align*}
    z_i^\ast
    =
    \frac{1}{2m}\left(\sum_{i=1}^{s}y_i\right),
    \qquad i=1,\ldots,2m.
\end{align*}
Consequently,
\begin{align*}
    a_m(K^{(s)})
    \ge
    \left(\sum_{i=m+1}^{2m} (z_i^\ast)^2\right)^{1/2}
    =
    \sqrt m\,
    \frac{1}{2m}\left(\sum_{i=1}^{s}y_i\right)
    \gtrsim
    \frac{\phi_K(s)}{\sqrt m}.
\end{align*}

Now apply Lemma \ref{lem:EN-threshold-implies-tail}.  With
$k_m=\lceil c_1m\log(ep/m)\rceil$,
\begin{align*}
    e_{k_m}(K^{(s)})
    \ge
    c_0 a_m(K^{(s)})
    \gtrsim
    \frac{\phi_K(s)}{\sqrt m}.
\end{align*}
Multiplying by $\sqrt{k_m}$ gives
\begin{align*}
    \sqrt{k_m}\,e_{k_m}(K^{(s)})
    \gtrsim
    \phi_K(s)\sqrt{\log(ep/m)}.
\end{align*}
Thus dropping the $\log$ factor we have,
\begin{align*}
    \sqrt{k_m}\,e_{k_m}(K^{(s)})
    \gtrsim
    \phi_K(s).
\end{align*}
Using the inequality from Lemma \ref{lem:width-order-stat} we have
\begin{align*}
    w(K^{(s)})
    \le
    A\,\phi_K(s)\sqrt{\log(ep)},
\end{align*}
we obtain
\begin{align*}
    \sqrt{k_m}\,e_{k_m}(K^{(s)})
    \gtrsim_{A,c}
    \frac{w(K^{(s)})}{\sqrt{\log(ep)}}.
\end{align*}

The reverse inequality follows from Sudakov's minoration.  Indeed, for every
$k\ge 1$,
\begin{align*}
    \sqrt{k}\,e_k(K^{(s)})
    \lesssim
    w(K^{(s)}).
\end{align*}
Combining the two bounds yields
\begin{align*}
    \sqrt{k_m}\,e_{k_m}(K^{(s)})
    \asymp_{\sqrt{\log ep}}
    w(K^{(s)})
\end{align*}
for all admissible $m$. 
This proves the claim.
\end{proof}

\begin{theorem}
    Suppose that $N \gtrsim \bigg(\frac{\sigma s \log (ep/s) }{w(K^{(s)})}\bigg)^2 \vee \frac{(w(K^{(s)}))^2 \ell_s}{\sigma^2} \vee \ell_s$, $N \lesssim p$ and it is known that $\beta \in K^{(s)}$. Let further $s \leq c_0 p$ for some sufficienlty small $c_0 > 0$. It follows that the estimator we proposed is minimax optimal up to logarithmic factor: $\log^2 (e p)\sqrt{\ell_s}$. 
\end{theorem}

\begin{remark} Whenever $N \gtrsim p$ for a sufficiently large constant, one can implement a minimax optimal estimator over the symmetric set $K^{(s)}$ in the way proposed in Section 3 of \cite{neykov2026fast}. 
\end{remark}

\begin{proof}
Write
\begin{align*}
    w_s:=w(K^{(s)}),
    \qquad
    h_s:=s\log(ep/s).
\end{align*}
By the minimax lower bound of
\cite{prasadan2025characterizingminimaxratenonparametric} (which notably does not place any assumption on $\sigma$ (unlike their upper bound which has to assume $\sigma$ is a fixed constant)), together with
Lemma 2.6 of \cite{neykov2026fast}, the minimax risk over $K^{(s)}$ is
bounded from below, up to logarithmic terms (i.e. $1 + \log_+ d_s \sqrt{N}/\sigma)$ which comes from an application of Lemma 2.6 of \cite{neykov2026fast}), by
\begin{align*}
    \left[
    \inf_{1\le k\le \lfloor p/2\rfloor-1}
    \left\{
        e_k(K^{(s)})^2+\frac{k\sigma^2}{N}
    \right\}
    \right]
    \wedge
    \frac{p\sigma^2}{N}
    \wedge
    d_s^2 .
\end{align*}

We first lower-bound the infimum term. Define the active entropy scale
\begin{align*}
    \kappa_s:=\frac{w_s\sqrt N}{\sigma},
\end{align*}
The assumption
\begin{align*}
    N\gtrsim
    \left(
        \frac{h_s\sigma}{w_s}
    \right)^2
\end{align*}
implies
\begin{align*}
    \kappa_s\gtrsim h_s.
\end{align*}
Using $\sigma \gtrsim \frac{w_s \sqrt{\ell_s}}{\sqrt{N}}$ gives $w_s \lesssim \sigma \frac{\sqrt{N}}{\sqrt{\ell_s}} \lesssim \sigma \sqrt{N} \lesssim \sigma \sqrt{p}$.
Hence we have $w_s \lesssim \sigma \sqrt{p}$.

Hence, since $N\lesssim p$,
\begin{align*}
    \kappa_s
    =
    \frac{w_s\sqrt N}{\sigma}
    \lesssim p.
\end{align*}
Thus $\kappa_s$ lies in the admissible entropy range.

Choose $m\ge s$ such that
\begin{align*}
    k_m:=\lceil c_1m\log(ep/m)\rceil
    \asymp
    \kappa_s .
\end{align*}
By Lemma \ref{lem:sudakov-tight-sparse-convexification},
\begin{align*}
    \sqrt{k_m}\,e_{k_m}(K^{(s)})
    \gtrsim
    \frac{w_s}{\sqrt{\log(ep)}}.
\end{align*}
Therefore
\begin{align*}
    e_{k_m}(K^{(s)})^2
    \gtrsim
    \frac{w_s^2}{k_m\log(ep)}
    \asymp
    \frac{\sigma w_s}{\sqrt N\,\log(ep)}.
\end{align*}

Now let
\begin{align*}
    F(k):=e_k(K^{(s)})^2+\frac{k\sigma^2}{N}.
\end{align*}
If $k\le k_m$, then, since $e_k(K^{(s)})$ is non-increasing in $k$,
\begin{align*}
    F(k)
    \ge
    e_k(K^{(s)})^2
    \ge
    e_{k_m}(K^{(s)})^2
    \gtrsim
    \frac{\sigma w_s}{\sqrt N\,\log(ep)}.
\end{align*}
If $k\ge k_m$, then
\begin{align*}
    F(k)
    \ge
    \frac{k\sigma^2}{N}
    \ge
    \frac{k_m\sigma^2}{N}
    \asymp
    \frac{\sigma w_s}{\sqrt N}.
\end{align*}
Combining the two cases yields
\begin{align*}
    \inf_{1\le k\le \lfloor p/2\rfloor-1}
    \left\{
        e_k(K^{(s)})^2+\frac{k\sigma^2}{N}
    \right\}
    \gtrsim
    \frac{\sigma w_s}{\sqrt N\,\log(ep)}.
\end{align*}

It remains to check that the truncation terms do not reduce this lower
bound. First,
\begin{align*}
    \frac{p\sigma^2}{N}
    \gtrsim
    \frac{\sigma w_s}{\sqrt N}
\end{align*}
is equivalent to
\begin{align*}
    \frac{w_s\sqrt N}{\sigma}\lesssim p,
\end{align*}
which was verified above. Second, by the lower active-scale assumption,
\begin{align*}
    \frac{\sigma w_s}{\sqrt N}
    \lesssim
    \frac{w_s^2}{h_s}
    \lesssim
    d_s^2,
\end{align*}
where the last inequality follows from
$w_s\lesssim d_s\sqrt{h_s}$. Therefore
\begin{align*}
    d_s^2
    \gtrsim
    \frac{\sigma w_s}{\sqrt N}.
\end{align*}
Thus the minimax risk over $K^{(s)}$ is bounded from below by
\begin{align*}
    \frac{\sigma w_s}{\sqrt N\,\log(ep)}.
\end{align*}

On the other hand, the oracle inequality for the estimator gives, for
$\beta\in K^{(s)}$,
\begin{align*}
    \mathbb E_\beta\|\widehat\beta-\beta\|_2^2
    \lesssim
    \frac{\sigma w_s\sqrt{\ell_s}}{\sqrt N}
    +
    \frac{\sigma w_s\ell_s}{N}+
    \frac{w_s^2\ell_s}{N}.
\end{align*}

Since
\begin{align*}
    \frac{w_s \sqrt{\ell_s}}{\sqrt{N}} \lesssim \sigma,
\end{align*}
we have
\begin{align*}
    \frac{w_s^2\ell_s}{N}
    \lesssim
    \frac{\sigma w_s\sqrt{\ell_s}}{\sqrt N}.
\end{align*}
Furthermore 
\begin{align*}
    \frac{w_s \ell_s}{N} \lesssim \frac{w_s \sqrt{\ell_s}}{\sqrt{N}},
\end{align*}
as $N \gtrsim \ell_s$.
Hence
\begin{align*}
    \mathbb E_\beta\|\widehat\beta-\beta\|_2^2
    \lesssim
    \frac{\sigma w_s\sqrt{\ell_s}}{\sqrt N}.
\end{align*}
Comparing the upper and lower bounds proves minimax optimality over
$K^{(s)}$ up to the logarithmic factor $\log^2(ep)\sqrt{\ell_s}$. The square in the $\log(ep)$ comes in since $ d_s \sqrt{N}/\sigma \lesssim w_s \sqrt{N}/\sigma \lesssim N/\sqrt{\ell_s} \lesssim p$ since $N \lesssim p$ (recall here that we had a term $1 + \log_+ (d_s \sqrt{N}/\sigma)$ from an application of the lower bound of \cite{prasadan2025characterizingminimaxratenonparametric}).
\end{proof}

\section{Discussion}\label{discussion:section}

This paper proposes a different way of thinking about sparse estimation under high-dimensional convex constraints.  Instead of beginning with a particular sparsity-inducing penalty, such as the $\ell_1$-norm, we start from an arbitrary sign-symmetric and permutation-invariant convex body $K$, and build from it the sparse convexification hierarchy
\begin{align*}
    K^{(s)}
    =
    \operatorname{conv}\{v\in K:\|v\|_0\le s\}.
\end{align*}
The resulting estimator searches over this hierarchy and adapts to the best sparse convex approximation of the target.  Thus sparsity enters not through a fixed regularizer, but through a sequence of convex relaxations of the original constraint.  This perspective gives a general oracle inequality for a broad class of symmetric norm constraints, assuming only oracle access to the Minkowski functional of $K$.

We showed near minimax optimality of our procedure for vectors $\beta \in K^{(s)}$ which includes $s$-sparse vectors for some sufficiently small $s$ compared to the sample size. It will be interesting to develop a regularized version of our procedure.

\bibliographystyle{abbrv}
\bibliography{polytime}

\newpage 

\appendix

\section{Supplemental Proofs}

We start by showing a lemma in greater generality than required. Turns out that one can compute the diameter $d_s$ of the set $K^{(s)}$ up to a logarithmic factor. This implies that one can also nearly compute the best  (i.e. smallest) $R$ such that $K \subseteq R B_2^p$ (upon applying the lemma to the set $K^{(p)} = K$). This lemma will in fact not be used, except to argue that $R$ can be approximated well up to a logarithmic loss.

\begin{lemma}[Sparse diameter via the fundamental function]\label{lem:sparse-diameter-fundamental}
Let $K\subseteq \mathbb R^p$ be a centrally symmetric convex body whose gauge
$\|\cdot\|_K$ is sign-invariant and permutation-invariant. For
$1\le s\le p$, define
\begin{align*}
    K^{(s)}
    :=
    \operatorname{conv}\{x\in K:\|x\|_0\le s\}.
\end{align*}
Let
\begin{align*}
    \rho_s
    :=
    \sup_{x\in K^{(s)}}\|x\|_2,
    \qquad
    d_s
    :=
    \operatorname{diam}_2(K^{(s)}).
\end{align*}
Also define the fundamental function
\begin{align*}
    \varphi_K(k)
    :=
    \left\|\sum_{i=1}^k e_i\right\|_K,
    \qquad 1\le k\le p,
\end{align*}
and
\begin{align*}
    r_s
    :=
    2\max_{1\le k\le s}
    \frac{\sqrt{k}}{\varphi_K(k)}.
\end{align*}
Then
\begin{align*}
    r_s/2
    \le
    \rho_s
    \le
    \sqrt{1+\lfloor \log_2 s\rfloor}\, r_s/2 .
\end{align*}
Equivalently, since $K^{(s)}$ is centrally symmetric and hence
$d_s=2\rho_s$,
\begin{align*}
    r_s
    \le
    d_s
    \le
    \sqrt{1+\lfloor \log_2 s\rfloor}\, r_s .
\end{align*}

Define $\tilde d_s = 2 R \wedge (r_s  \sqrt{1+\lfloor \log_2 s\rfloor})$.
\end{lemma}

\begin{proof}
Since $\|\cdot\|_2$ is convex and $K^{(s)}$ is the convex hull of the
$s$-sparse points of $K$, we have
\begin{align*}
    \rho_s
    =
    \sup_{x\in K^{(s)}}\|x\|_2
    =
    \sup_{\substack{x\in K\\ \|x\|_0\le s}}\|x\|_2 .
\end{align*}

We first prove the lower bound. Fix $1\le k\le s$. By definition of
$\varphi_K(k)$, the vector
\begin{align*}
    x^{(k)}
    :=
    \frac{1}{\varphi_K(k)}
    \sum_{i=1}^k e_i
\end{align*}
belongs to $K$ and is $k$-sparse. Hence $x^{(k)}\in K^{(s)}$, and
\begin{align*}
    \rho_s
    \ge
    \|x^{(k)}\|_2
    =
    \frac{\sqrt{k}}{\varphi_K(k)}.
\end{align*}
Taking the maximum over $1\le k\le s$ gives
\begin{align*}
    \rho_s\ge r_s/2.
\end{align*}

We now prove the upper bound. Let $x\in K$ with $\|x\|_0\le s$. By sign
and permutation invariance, we may assume without loss of generality that
\begin{align*}
    x_1\ge x_2\ge \cdots \ge x_s\ge 0,
    \qquad
    x_{s+1}=\cdots=x_p=0.
\end{align*}
For every $1\le k\le s$, the vector
\begin{align*}
    x_k\sum_{i=1}^k e_i
\end{align*}
is coordinatewise dominated by $x$. Since $\|\cdot\|_K$ is unconditional,
it is monotone with respect to coordinatewise domination. Therefore
\begin{align*}
    x_k \varphi_K(k)
    =
    \left\|x_k\sum_{i=1}^k e_i\right\|_K
    \le
    \|x\|_K
    \le 1.
\end{align*}
Thus
\begin{align*}
    x_k
    \le
    \frac{1}{\varphi_K(k)}
    \qquad
    \text{for every }1\le k\le s.
\end{align*}

Let $m=\lfloor \log_2 s\rfloor$. Decompose $\{1,\dots,s\}$ into dyadic
blocks
\begin{align*}
    B_j
    :=
    \{2^j,2^j+1,\dots,\min(2^{j+1}-1,s)\},
    \qquad j=0,\dots,m.
\end{align*}
Since $x_1\ge \cdots \ge x_s\ge 0$, for every $i\in B_j$,
\begin{align*}
    x_i\le x_{2^j}.
\end{align*}
Also $|B_j|\le 2^j$. Hence
\begin{align*}
    \sum_{i\in B_j}x_i^2
    \le
    2^j x_{2^j}^2.
\end{align*}
Summing over the dyadic blocks gives
\begin{align*}
    \|x\|_2^2
    =
    \sum_{i=1}^s x_i^2
    \le
    \sum_{j=0}^{m} 2^j x_{2^j}^2.
\end{align*}
Using the previous bound with $k=2^j$, we obtain
\begin{align*}
    2^j x_{2^j}^2
    \le
    \frac{2^j}{\varphi_K(2^j)^2}
    \le
    (r_s/2)^2.
\end{align*}
Therefore
\begin{align*}
    \|x\|_2^2
    \le
    \sum_{j=0}^{m} (r_s/2)^2
    =
    \bigl(1+\lfloor \log_2 s\rfloor\bigr)(r_s/2)^2.
\end{align*}
Taking the supremum over all $x\in K$ with $\|x\|_0\le s$, we get
\begin{align*}
    \rho_s
    \le
    \sqrt{1+\lfloor \log_2 s\rfloor}\,r_s/2.
\end{align*}

Finally, since $K^{(s)}$ is centrally symmetric,
\begin{align*}
    d_s
    =
    \operatorname{diam}_2(K^{(s)})
    =
    2\sup_{x\in K^{(s)}}\|x\|_2
    =
    2\rho_s.
\end{align*}
This proves the claimed bounds for $d_s$.
\end{proof}

\begin{proof}[Proof of Lemma \ref{lem:width-order-stat}]
By the support-function identity for $K^{(s)}$,
\begin{align*}
    h_{K^{(s)}}(g)
    =
    \max_{|T|\le s}\|g_T\|_{K^\circ}.
\end{align*}
Since $K^\circ$ is also sign-invariant and permutation-invariant, the
maximum is attained by taking $T$ to be the indices of the $s$ largest
coordinates of $g$ in absolute value. Hence
\begin{align*}
    h_{K^{(s)}}(g)
    =
    \left\|
    (g_1^*,\ldots,g_s^*,0,\ldots,0)
    \right\|_{K^\circ}.
\end{align*}
Taking expectations gives the identity for the Gaussian width.

We use the following standard bound for Gaussian order statistics.  There is a
universal constant $C>0$ such that, for every $u\ge 0$, with probability at
least $1-e^{-u}$,
\begin{align*}
    g_j^*
    \le
    C\left(\sqrt{\log(ep/j)}+\sqrt u\right),
    \qquad 1\le j\le s .
\end{align*}
Indeed, fix $1\le j\le s$ and set
\begin{align*}
    t_j := A\left(\sqrt{\log(ep/j)}+\sqrt u\right),
\end{align*}
where $A>0$ is a sufficiently large numerical constant.  Since
$\{g_j^*>t_j\}$ implies that at least $j$ coordinates satisfy
$|g_i|>t_j$, the Gaussian tail bound gives
\begin{align*}
\begin{aligned}
    \mathbb P(g_j^*>t_j)
    &\le
    {p\choose j}\mathbb P(|g_1|>t_j)^j  \\
    &\le
    \left(\frac{ep}{j}\right)^j
    \left(2e^{-t_j^2/2}\right)^j .
\end{aligned}
\end{align*}
Writing $L_j=\log(ep/j)$, and using
\begin{align*}
    t_j^2
    \ge
    A^2(L_j+u),
\end{align*}
we obtain
\begin{align*}
    \mathbb P(g_j^*>t_j)
    \le
    \left[
        2\exp\{(1-A^2/2)L_j\}
        \exp\{-A^2u/2\}
    \right]^j .
\end{align*}
Since $L_j\ge 1$, choosing $A$ large enough yields
\begin{align*}
    \mathbb P(g_j^*>t_j)
    \le
    e^{-2j}e^{-4uj}.
\end{align*}
Therefore,
\begin{align*}
    \mathbb P\left(
        \exists\,1\le j\le s:
        g_j^*>A\left(\sqrt{\log(ep/j)}+\sqrt u\right)
    \right)
    \le
    \sum_{j=1}^s e^{-2j}e^{-4uj}
    \le
    e^{-u}.
\end{align*}
This proves the claimed simultaneous bound.
Therefore, by monotonicity of the symmetric norm,
\begin{align*}
    \left\|
    (g_1^*,\ldots,g_s^*,0,\ldots,0)
    \right\|_{K^\circ}
    \le
    C\|\gamma^{(s)}\|_{K^\circ}
    +
    C\sqrt u\,\left\|\sum_{j=1}^s e_j\right\|_{K^\circ}.
\end{align*}
Since $\sqrt{\log(ep/j)}\ge 1$ for $j\le s$, we have
\begin{align*}
    \left\|\sum_{j=1}^s e_j\right\|_{K^\circ}
    \le
    \|\gamma^{(s)}\|_{K^\circ}.
\end{align*}
Integrating the tail bound in $u$ gives
\begin{align*}
    w(K^{(s)})
    \le
    C\|\gamma^{(s)}\|_{K^\circ}.
\end{align*}

\end{proof}

\begin{proof}[Proof of Lemma \ref{lem:matching-lower-bound-gaussian-width}]
Let
\begin{align*}
    X^{(s)} := (g_1^*,\ldots,g_s^*,0,\ldots,0).
\end{align*}
By the support-function identity,
\begin{align*}
    w(K^{(s)})
    =
    \mathbb E\|X^{(s)}\|_{K^\circ}.
\end{align*}

We first record a lower bound on partial sums of Gaussian order statistics.
For $1\le k\le p$, set
\begin{align*}
    S_k := \sum_{j=1}^k g_j^*,
    \qquad
    \Gamma_k := \sum_{j=1}^k \sqrt{\log(ep/j)}.
\end{align*}
We claim that
\begin{align*}
    \mathbb E S_k \ge c \Gamma_k ,
    \qquad 1\le k\le p.
\end{align*}

Indeed, first note that
\begin{align*}
    \Gamma_k
    \le
    C k\sqrt{\log(ep/k)} .
\end{align*}
This follows from concavity of $x\mapsto \sqrt{x}$ and the bound
\begin{align*}
    \frac1k\sum_{j=1}^k \log(ep/j)
    =
    \log(ep)-\frac1k\log(k!)
    \le
    \log(ep/k)+1
    \le
    2\log(ep/k),
\end{align*}
which is true by standard integration, i.e., $\log k! \geq \int_{1}^k \log x dx = k \log k - k + 1$.

Now put
\begin{align*}
    t := a\sqrt{\log(ep/k)}
\end{align*}
for a sufficiently small universal constant $a>0$, and define
\begin{align*}
    N_t := \#\{1\le i\le p: |g_i|\ge t\}.
\end{align*}
Then
\begin{align*}
    S_k \ge t\,(N_t\wedge k).
\end{align*}
Let $q_t:=\mathbb P(|g_1|\ge t)$.  The standard Gaussian lower tail bound
\begin{align*}
    \mathbb P(|g_1|\ge t)
    \ge
    \frac{c}{1+t^2}e^{-t^2/2}
\end{align*}
implies, for $a>0$ small enough, that
\begin{align*}
    p q_t \ge \eta k
\end{align*}
for a universal constant $\eta>0$.  Hence
\begin{align*}
    \mathbb E N_t = p q_t \ge \eta k.
\end{align*}
By Paley--Zygmund, since $N_t$ is binomial,
\begin{align*}
    \mathbb P\left(N_t\ge \frac12\mathbb E N_t\right)
    \ge c_\eta
\end{align*}
for a universal constant $c_\eta>0$.  Therefore
\begin{align*}
    \mathbb E(N_t\wedge k) \ge \mathbb E((N_t\wedge k)\mathbbm{1}(N_t \geq \EE N_t/2)) 
    \ge (\EE N_t/2\wedge k )\PP(N_t \geq \EE N_t/2) \geq
    c k.
\end{align*}
Consequently,
\begin{align*}
    \mathbb E S_k
    \ge
    c k \sqrt{\log(ep/k)}
    \ge
    c \Gamma_k .
\end{align*}

Now let
\begin{align*}
    m := \mathbb E X^{(s)}
    =
    \left(
    \mathbb E g_1^*,\ldots,\mathbb E g_s^*,0,\ldots,0
    \right).
\end{align*}
The previous partial-sum estimate gives, for every $1\le k\le p$,
\begin{align*}
    \sum_{j=1}^k m_j
    \ge
    c\sum_{j=1}^k \gamma^{(s)}_j .
\end{align*}
Thus $c\gamma^{(s)}$ is weakly submajorized by $m$.  Since
$\|\cdot\|_{K^\circ}$ is sign-invariant and permutation-invariant, the
Ky Fan dominance principle for symmetric norms gives
\begin{align*}
    \|m\|_{K^\circ}
    \ge
    c\|\gamma^{(s)}\|_{K^\circ}.
\end{align*}
Finally, by Jensen's inequality,
\begin{align*}
    w(K^{(s)})
    =
    \mathbb E\|X^{(s)}\|_{K^\circ}
    \ge
    \|\mathbb E X^{(s)}\|_{K^\circ}
    =
    \|m\|_{K^\circ}.
\end{align*}
Combining the last two displays yields
\begin{align*}
    w(K^{(s)})
    \ge
    c\|\gamma^{(s)}\|_{K^\circ}.
\end{align*}
\end{proof}

\begin{proof}[Proof of Lemma \ref{lem:weak-membership-Ks-polar}]
Let
\begin{align*}
    A := (K^{(s)})^\circ,
    \qquad
    D := K^\circ,
\end{align*}
and set
\begin{align*}
    T=T_s(z),
    \qquad
    y=z_T,
\end{align*}
By Lemma~\ref{lem:Ks-polar},
\begin{align*}
    z\in A
    \quad\Longleftrightarrow\quad
    z_{T_s(z)}\in D .
\end{align*}
Thus, if $z\in A$, then $y=z_{T_s(z)}\in D$, and the weak membership
oracle for $D$ returns $\textnormal{YES}$.

It remains to prove the contrapositive of the $\textnormal{NO}$ guarantee.
Assume that
\begin{align*}
    \operatorname{dist}(y,D)\le \delta .
\end{align*}
Choose $y_0\in D$ such that
\begin{align*}
    \|y-y_0\|_2\le \delta .
\end{align*}
Since
\begin{align*}
    \frac1R B_2^p\subseteq D=K^\circ,
\end{align*}
the Minkowski functional of $D$ satisfies
\begin{align*}
    \|u\|_D\le R\|u\|_2
    \qquad\text{for all }u\in\mathbb R^p.
\end{align*}
Therefore
\begin{align*}
    \|y\|_D
    \le
    \|y_0\|_D+\|y-y_0\|_D
    \le
    1+R\delta .
\end{align*}
Using again the identity from Lemma~\ref{lem:Ks-polar},
\begin{align*}
    \|z\|_A
    =
    \|z_{T_s(z)}\|_{K^\circ}
    =
    \|y\|_D
    \le
    1+R\delta .
\end{align*}
Let
\begin{align*}
    \alpha := 1+R\delta .
\end{align*}
Then $z/\alpha\in A$. Moreover, by Lemma~\ref{lem:Ks-polar},
\begin{align*}
    A=(K^{(s)})^\circ
    \subseteq
    b_s B_2^p,
    \qquad
    b_s=\frac1r\sqrt{\frac{p}{s}} .
\end{align*}
Since $z/\alpha\in A$, we have
\begin{align*}
    \|z\|_2\le \alpha b_s .
\end{align*}
Hence
\begin{align*}
    \operatorname{dist}(z,A)
    \le
    \left\|z-\frac{z}{\alpha}\right\|_2
    =
    \left(1-\frac1\alpha\right)\|z\|_2
    \le
    (\alpha-1)b_s
    =
    R\delta b_s
    =
    \varepsilon .
\end{align*}
Thus, if $\operatorname{dist}(z,A)>\varepsilon$, then necessarily
\begin{align*}
    \operatorname{dist}(z_{T_s(z)},K^\circ)>\delta ,
\end{align*}
so the weak membership oracle for $K^\circ$ returns
$\textnormal{NO}$. This proves the claim.
\end{proof}

\begin{proof}[Proof of Lemma \ref{lem:polar-weak-membership}]
Since
\begin{align*}
    rB_2^p\subseteq C\subseteq RB_2^p,
\end{align*}
the polar body satisfies
\begin{align*}
    \frac1R B_2^p \subseteq C^\circ \subseteq \frac1r B_2^p.
\end{align*}
Thus $C^\circ$ is also well balanced, with known inner and outer radii.
By the standard equivalence between weak membership, weak separation, and
weak optimization for well-balanced convex bodies, a weak membership oracle
for $C^\circ$ yields a polynomial-time weak optimization oracle over
$C^\circ$ \cite{grotschel2012geometric}. In particular, for any $x\in\mathbb R^p$ and any accuracy
$\eta>0$, we can compute a number $\widehat h(x)$ satisfying
\begin{align*}
    \bigl|\widehat h(x)-h_{C^\circ}(x)\bigr|\le \eta,
\end{align*}
where
\begin{align*}
    h_{C^\circ}(x)
    :=
    \sup_{y\in C^\circ}\langle y,x\rangle
\end{align*}
is the support function of $C^\circ$.

By the bipolar theorem,
\begin{align*}
    C = (C^\circ)^\circ,
\end{align*}
and hence
\begin{align*}
    x\in C
    \quad\Longleftrightarrow\quad
    h_{C^\circ}(x)\le 1.
\end{align*}
We now show that points at Euclidean distance more than $\varepsilon$ from
$C$ have a quantitative gap in support function value. Suppose
\begin{align*}
    d:=\operatorname{dist}(x,C)>\varepsilon.
\end{align*}
Let $z\in C$ be the Euclidean projection of $x$ onto $C$, and set
\begin{align*}
    u:=\frac{x-z}{\|x-z\|_2}.
\end{align*}
By the projection optimality condition,
\begin{align*}
    \langle x-z,w-z\rangle\le 0
    \qquad\text{for all } w\in C.
\end{align*}
Equivalently,
\begin{align*}
    \langle u,w\rangle\le \langle u,z\rangle
    \qquad\text{for all } w\in C.
\end{align*}
Therefore
\begin{align*}
    h_C(u)=\sup_{w\in C}\langle u,w\rangle=\langle u,z\rangle.
\end{align*}
Since $rB_2^p\subseteq C\subseteq RB_2^p$ and $\|u\|_2=1$,
\begin{align*}
    r\le h_C(u)\le R.
\end{align*}
Define
\begin{align*}
    y_0:=\frac{u}{h_C(u)}.
\end{align*}
Then $y_0\in C^\circ$, because for every $w\in C$,
\begin{align*}
    \langle y_0,w\rangle
    =
    \frac{\langle u,w\rangle}{h_C(u)}
    \le 1.
\end{align*}
Moreover,
\begin{align*}
    \langle y_0,x\rangle
    =
    \frac{\langle u,z\rangle+\langle u,x-z\rangle}{h_C(u)}
    =
    \frac{h_C(u)+d}{h_C(u)}
    =
    1+\frac{d}{h_C(u)}
    \ge
    1+\frac{d}{R}
    >
    1+\frac{\varepsilon}{R}.
\end{align*}
Hence
\begin{align*}
    h_{C^\circ}(x)>1+\frac{\varepsilon}{R}.
\end{align*}

Now choose
\begin{align*}
    \eta := \frac{\varepsilon}{4R}.
\end{align*}
Using the weak optimization oracle over $C^\circ$, compute $\widehat h(x)$ such that
\begin{align*}
    \bigl|\widehat h(x)-h_{C^\circ}(x)\bigr|\le \eta.
\end{align*}
Return
\begin{align*}
    \textnormal{YES}
    \quad\text{if}\quad
    \widehat h(x)\le 1+\frac{\varepsilon}{2R},
\end{align*}
and return
\begin{align*}
    \textnormal{NO}
    \quad\text{otherwise}.
\end{align*}

If $x\in C$, then $h_{C^\circ}(x)\le 1$, so
\begin{align*}
    \widehat h(x)\le 1+\eta
    =
    1+\frac{\varepsilon}{4R}
    <
    1+\frac{\varepsilon}{2R}.
\end{align*}
Thus the procedure returns $\textnormal{YES}$.

On the other hand, if $\operatorname{dist}(x,C)>\varepsilon$, then
\begin{align*}
    h_{C^\circ}(x)>1+\frac{\varepsilon}{R}.
\end{align*}
Therefore
\begin{align*}
    \widehat h(x)
    \ge
    h_{C^\circ}(x)-\eta
    >
    1+\frac{\varepsilon}{R}-\frac{\varepsilon}{4R}
    =
    1+\frac{3\varepsilon}{4R}
    >
    1+\frac{\varepsilon}{2R}.
\end{align*}
Thus the procedure returns $\textnormal{NO}$.

The output may be arbitrary when $\operatorname{dist}(x,C)\le\varepsilon$,
which is exactly the allowed boundary band for weak membership. The running
time is polynomial in $p$, $\log(R/r)$, and $\log(1/\varepsilon)$, and
uses only the weak membership oracle for $C^\circ$.
\end{proof}

\begin{lemma}\label{lem:diameter-penalty-quasi-subadditive}
    We have $
    d_{s+t}^2\le d_s^2+d_t^2, s+t\le p$ for any $s + t \leq p$. In particular $d_s^2 \lesssim s d_1^2$.
\end{lemma}

\begin{proof}[Proof of Lemma \ref{lem:diameter-penalty-quasi-subadditive}]
Let
\begin{align*}
    \rho_s:=\sup_{x\in K^{(s)}}\|x\|_2=\frac{d_s}{2}.
\end{align*}
We first prove that, whenever $s+t\le p$,
\begin{align*}
    \rho_{s+t}^2\le \rho_s^2+\rho_t^2.
\end{align*}
Since $K^{(s+t)}$ is the convex hull of
\begin{align*}
    S_{s+t}:=\{x\in K:\|x\|_0\le s+t\},
\end{align*}
and $x\mapsto \|x\|_2^2$ is convex,
\begin{align*}
    \sup_{x\in K^{(s+t)}}\|x\|_2^2
    =
    \sup_{x\in S_{s+t}}\|x\|_2^2.
\end{align*}
Fix $x\in S_{s+t}$. Split its support into two disjoint sets
$T_1,T_2$ such that
\begin{align*}
    |T_1|\le s,\qquad |T_2|\le t,
    \qquad \operatorname{supp}(x)\subseteq T_1\cup T_2.
\end{align*}
Because $K$ is sign-invariant and convex, coordinate projections preserve
membership in $K$. Hence
\begin{align*}
    x_{T_1}\in K,
    \qquad
    x_{T_2}\in K.
\end{align*}
Therefore
\begin{align*}
    x_{T_1}\in K^{(s)},
    \qquad
    x_{T_2}\in K^{(t)}.
\end{align*}
Since the supports are disjoint,
\begin{align*}
    \|x\|_2^2
    =
    \|x_{T_1}\|_2^2+\|x_{T_2}\|_2^2
    \le
    \rho_s^2+\rho_t^2.
\end{align*}
Taking the supremum over $x\in S_{s+t}$ gives
\begin{align*}
    \rho_{s+t}^2\le \rho_s^2+\rho_t^2.
\end{align*}
Equivalently,
\begin{align*}
    d_{s+t}^2\le d_s^2+d_t^2,
    \qquad s+t\le p.
\end{align*}

Thus we conclude that $d_{s}^2 \leq s d_1^2$. 
\end{proof}

\begin{lemma}[Gaussian width of $K^{(s)}$]\label{lem:Ks-width}
For every $1\le s\le p$,
\begin{align*}
    w(K^{(s)})
    \le
    C d_s\sqrt{s\log(ep/s)}.
\end{align*}
More precisely, since $K^{(s)}$ is centrally symmetric,
\begin{align*}
    w(K^{(s)})
    \le
    C\operatorname{rad}(K^{(s)})\sqrt{s\log(ep/s)}
    =
    C d_s\sqrt{s\log(ep/s)},
\end{align*}
where the last equality is up to an absolute factor.
\end{lemma}

\begin{proof}[Proof of Lemma \ref{lem:Ks-width}]
Because a linear functional has the same supremum over a set and over its convex hull,
\begin{align*}
    \sup_{v\in K^{(s)}}\langle g,v\rangle
    =
    \sup_{v\in S_s}\langle g,v\rangle .
\end{align*}
For $v\in S_s$, if $T=\operatorname{supp}(v)$, then $|T|\le s$ and
\begin{align*}
    \langle g,v\rangle
    =
    \langle g_T,v_T\rangle
    \le
    \|g_T\|_2\|v\|_2
    \le
    \operatorname{rad}(K^{(s)})\|g_T\|_2.
\end{align*}
Therefore
\begin{align*}
    \sup_{v\in K^{(s)}}\langle g,v\rangle
    \le
    \operatorname{rad}(K^{(s)})
    \max_{|T|\le s}\|g_T\|_2.
\end{align*}
It remains to use the standard sparse-Gaussian bound
\begin{align*}
    \EE\max_{|T|\le s}\|g_T\|_2
    \le
    C\sqrt{s\log(ep/s)}.
\end{align*}
For completeness, this follows by a union bound over the at most $(ep/s)^s$ supports of size $s$ and the usual concentration of $\chi_s$ random variables. See also Lemma 5.14 of \cite{prasadan2026information} for a similar argument. Combining the two displays proves the claim.
\end{proof}

\end{document}